\documentclass[11pt]{article}

\usepackage{amssymb,amsmath,amsthm,mathtools}
\usepackage{graphicx}
\usepackage{subcaption}
\usepackage[toc,page]{appendix}
\usepackage{xcolor}
\usepackage{tikz}
\usepackage{dsfont}
\usetikzlibrary{positioning}
\usepackage{tkz-euclide}
\usetikzlibrary{arrows,shapes,positioning,plotmarks}
\usepackage{pgfplots}

\usepackage{authblk}
\usepackage{hyperref}

\mathtoolsset{showonlyrefs=true}

\newtheorem{corollary}{Corollary}
\newtheorem{proposition}{Proposition}
\newtheorem{lemma}{Lemma}
\newtheorem{theorem}{Theorem}
\newtheorem{remark}{Remark}
\newtheorem{conjecture}{Conjecture}

\newcommand{\pr}[1]{{\mathbb P}\{#1\}}

\newcommand{\keywords}[1]{\textit{Keywords}: #1}

\newcommand{\noopsort}[1]{}  

\title{Power-of-two sampling in redundancy systems:\\ the impact of assignment constraints}

\author[a]{Ellen~Cardinaels\thanks{Corresponding author. E-mail address: \href{mailto:e.cardinaels@tue.nl}{e.cardinaels@tue.nl}}}
\author[a]{Sem~Borst}
\author[b]{Johan~S.H.~van~Leeuwaarden}
\affil[a]{\small \textit{Eindhoven University of Technology, The Netherlands}}
\affil[b]{\small \textit{Tilburg University, The Netherlands}}

\usepackage{titling}
\predate{}
\postdate{}
\date{} 

\begin{document}
\maketitle

\begin{abstract}
A classical sampling strategy for load balancing policies is power-of-two, where any server pair is sampled with equal probability. This does not cover practical settings with assignment constraints which force non-uniform sampling. While intuition suggests that non-uniform sampling adversely impacts performance, this was only supported through simulations, and rigorous statements have remained elusive.  Building on product-form distributions for redundancy systems, we prove the stochastic dominance of uniform sampling for a four-server system as well as arbitrary-size systems in light traffic.
\end{abstract}

\keywords power-of-two, parallel-server systems, load balancing, redundancy scheduling, light traffic, stochastic comparison.

\section{Introduction} 
Load balancing policies applying a power-of-$d$ sampling strategy assign a job to the `best' among a randomly selected subset of ${d \geq 2}$ parallel servers. These policies were originally investigated in a balls-and-bins context where it was shown that the maximum bin occupancy is exponentially reduced when instead of purely random assignment the least loaded bin among the $d$ selected bins is chosen~\cite{Azar1999}.
This concept was also shown to be highly effective in queueing contexts, in particular in the many-server regime where the above policy results in a doubly exponential improvement in terms of the queue length distribution per server compared to purely random assignment~\cite{mitzenmacher2001power,vvedenskaya1996queueing}. Moreover, power-of-$d$ policies only involve a low implementation overhead and hence provide scalability, which is critical in large-scale systems such as data centers. More recently, power-of-$d$ sampling policies have also been proposed in the context of redundancy scheduling where replicas of each job are dispatched to a randomly selected subset of $d$~servers~\cite{gardner2017redundancy}. We refer to~\cite{Boor2018} for further background on scalable load balancing algorithms.

In the classical power-of-$d$ setup, it is implicitly assumed that the $d$~servers are selected uniformly at random, with or without replacement. This is a natural assumption when all jobs and servers are mutually exchangeable, and also mathematically convenient (see the discussion of related literature below for further details).
However, this assumption excludes situations with assignment constraints, such as locality constraints or compatibility relations between jobs and servers, which force non-uniform server sampling.
In the special case that $d=2$ such assignment constraints can conveniently be visualized in a graph consisting of $N$~nodes, one for each server. When a non-negative weight is associated with each edge, edges or server pairs can be sampled according to these weights. When server pairs are sampled uniformly at random, all edges receive a weight equal to $1/\binom{N}{2}$. This raises the following interesting question: \emph{How do non-uniform edge weights affect the system performance, as compared to uniform edge weights?} One intuitively  expects performance to benefit from more flexibility (i.e., more non-zero weights) and homogeneity (i.e., uniform edge weights). This intuition was supported through heuristic arguments and simulations for the Join-the-Shortest-Queue policy in specific topological settings, see for instance the thesis of Mitzenmacher~\cite{Mitzenmacher_thesis}, the seminal paper of Turner~\cite{Turner1998} and the more recent work of Gast~\cite{gast2015power}.
However to the best of our knowledge, rigorous statements on the performance impact of assignment constraints in the power-of-choice setting have remained elusive so far.

In the present paper we establish stochastic comparison results which corroborate the above-mentioned `common wisdom' in redundancy systems, and prove in some specific settings that the classical uniform power-of-two policy outperforms \textit{any} power-of-two policy with assignment constraints.
To establish these results we employ the product-form expressions for the stationary  occupancy distribution obtained by Gardner {\em et al.}~\cite{Gardner2016queueing}. Unfortunately, the detailed job-level state description yields expressions that do not give immediate insight into the system performance. Careful inspection and further manipulation of the detailed product-form expressions, however, allows us to derive stochastic comparison results.

We first establish closed-form expressions for the stationary distribution of the total number of jobs for four-server systems, which we then use to show a stochastic ordering result for a ring graph compared to a complete graph, confirming the above intuition. For systems of arbitrary size, closed-form expressions for the stationary distribution of the total number of jobs in the system seem out of reach. However, focusing on a light-traffic scenario allows us to extract the essential information to compare the stationary distributions of systems with different edge selection probabilities. This comparison gives rise to an optimization problem in terms of the edge selection probabilities for which the classical uniform power-of-two policy arises as the optimal solution. Moreover, the light-traffic comparison can be interpreted as a design guideline for an efficient weighted power-of-two policy in the presence of assignment constraints. 

The literature focusing on the classical uniform power-of-$d$ sampling policy is extensive and vibrant as the inherent symmetry of these policies lends itself well to asymptotic analysis in a many-server regime.
Seminal results in such settings were obtained by Mitzenmacher~\cite{mitzenmacher2001power} and Vvedenskaya {\em et al.}~\cite{vvedenskaya1996queueing} using fluid-limit techniques, and later closely related mean-field concepts were studied in~\cite{gardner2017redundancy,Hellemans2019,Hellemans2018,
Hellemans2021}.

As in~\cite{Jaleel2021,Gardner2021,Zhan2018}, these techniques can also be used for the analysis of particular asymmetric assignment constraints where mutually exchangeable servers are clustered in a finite number of pools.
However, fluid-limit and mean-field techniques are usually not well-suited to scenarios with asymmetric assignment constraints corresponding to a graph as mentioned above.
Indeed, Gast~\cite{gast2015power} and Turner~\cite{Turner1998} use an approximation scheme and simulations to demonstrate that the classical Join-the-Shortest-Queue(2) (JSQ(2)) policy outperforms a restricted JSQ(2) policy where the assignment is governed by a ring graph. In contrast, the results in \cite{Budhiraja2019,Mukherjee2018,Rutten2020,Weng2020} establish conditions in terms of the assignment constraints that yield performance comparable to the classical uniform power-of-$d$ policies in a many-server regime.

Non-uniform selection of subsets of servers in a JSQ context has also been considered by He and Down~\cite{He2008} and Sloothaak {\em et al.}~\cite{Sloothaak2021}, where it is shown that the diffusion scaled queue length process coincides with that of a fully pooled system in a heavy-traffic regime. In \cite{Cardinaels2022} conditions in terms of the assignment constraints are established to draw a similar conclusion for redundancy policies.
Rather than imposing conditions on the assignment constraints, our results aim to connect and compare the performance of systems operating under the classical uniform power-of-two sampling policies and those with assignment constraints. We will focus on systems with a fixed number of servers in moderate or light traffic.

The remainder of this paper is organized as follows. In Section~\ref{sec:model_description} we
present a detailed model description and discuss some broader context and preliminaries. In Subsection~\ref{subsec:intro_example} we set out to prove a stochastic comparison between two small systems. Next we consider systems of arbitrary size in Subsection~\ref{sec:main_results}, and establish a light-traffic comparison between the classical uniform power-of-two policy and weighted power-of-two policies.
A discussion of the results and some pointers for further research are provided in Section~\ref{sec:extensions}. 

\section{Model description and preliminaries} \label{sec:model_description}

\subsection{Model description}
Before elaborating on redundancy scheduling, we first define the power-of-two policies to sample server pairs subject to the assignment constraints. 
Jobs arrive according to a Poisson process with rate~$N\lambda$, with $N$ the total number of parallel servers.
When a job arrives, the server pair available for its assignment is $\{i,j\}$ with probability $p_{\{i,j\}}$, $i,j,\in\{1,\dots,N\}$ and $i\neq j$. For simplicity we refer to such a job as a type-$\{i,j\}$ job. Due to the properties of the Poisson process, one can equivalently take the view that type-$\{i,j\}$ jobs arrive according to a Poisson process with rate $N\lambda p_{\{i,j\}}$. 

As alluded to in the introduction, one can think of an underlying simple graph structure with $N$~nodes and (non-negative) edge weight $p_{\{i,j\}}$ for the edge $\{i,j\}$. Server pairs are then sampled proportionally to these edge weights for each arriving job.
 Let $\mathcal{E} \coloneqq \{ \{i,j\} \mid p_{\{i,j\}}>0, i,j=1,\dots,N, i\neq j\}$ denote the set of all edges with a non-zero weight, or alternatively, all possible job types that can occur in the system. Without loss of generality we assume that $\sum_{e\in\mathcal{E}}p_e = 1$, and refer to $p_e$ as the selection probability of the edge $e$. An example where this underlying graph is given by a ring graph is depicted in Figure~\ref{fig:example1}.

The setting where $p_{\{i,j\}} \equiv 1/E$ for all $i\neq j$, with $E=\binom{N}{2}$ the total number of different server pairs, corresponds to the typical power-of-two setting with uniform sampling. We will henceforth refer to this setting as the \textit{classical} power-of-two policy, while any setting with non-uniform sampling is referred to as a \textit{weighted} power-of-two policy.

\begin{remark}\upshape
When jobs can be assigned to $d\ge2$~servers, one can think of $p_e$ as the selection probability of hyper-edge $e$ in a hypergraph with $N$ nodes where each hyper-edge is incident to $d$ distinct nodes. In the remainder of this paper we will focus on the case where $d=2$.
\end{remark}

\begin{figure}[h]

\begin{subfigure}[t]{0.25\textwidth}
  \centering
\includegraphics{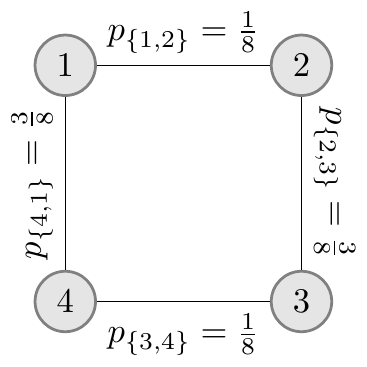}
  \caption{A ring as an underlying graph structure.}
  \label{fig:example1}
\end{subfigure}%
~
\begin{subfigure}[t]{.37\textwidth}
  \centering
  \includegraphics{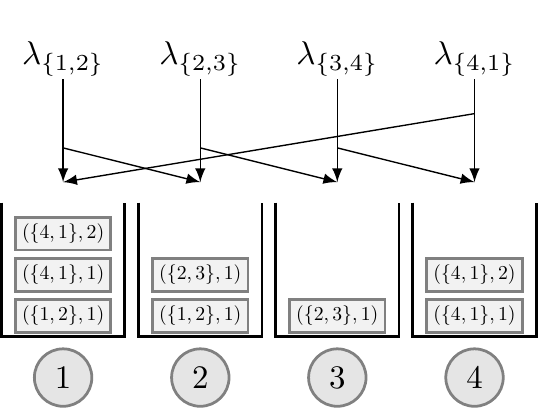}
  \caption{Representation with one queue per server and replicas.}
  \label{fig:example2}
\end{subfigure}
\begin{subfigure}[t]{0.37\textwidth}
  \centering 
  \includegraphics{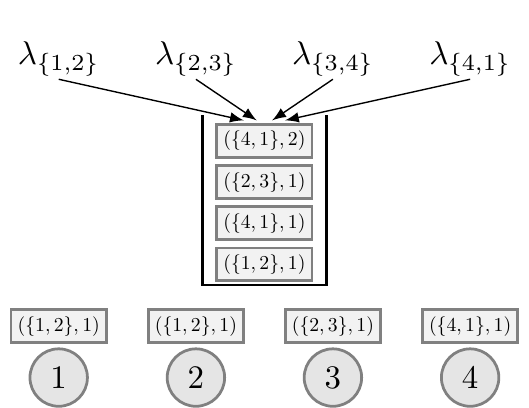}
  \caption{Central queue representation.}
  \label{fig:example3}
\end{subfigure}
\caption{With the underlying graph structure as indicated in Panel~(a) with ${N=4}$ servers, Panels~(b) and~(c) give two different representations of the state ${\boldsymbol{c} = (\{1,2\},\{4,1\},\{2,3\},\{4,1\})}$. The notation $(\{i,j\},k)$ stands for the $k$th arrival of a type-$\{i,j\}$ job.}
\label{fig:example}
\end{figure}

For an arriving type-$\{i,j\}$ job, under the redundancy policy, replicas are assigned to both servers $i$ and $j$, with $i,j,\in\{1,\dots,N\}$ and $i\neq j$. The service requirements of the two replicas are independent and exponentially distributed with unit mean. Each server has speed $\mu>0$ and handles the assigned jobs in a First-Come-First-Served manner. Once the first replica finishes service, the remaining replica will be discarded instantaneously. 

As discussed in the introduction, it is intuitively plausible that uniform sampling outperforms non-uniform sampling. This intuitive notion is formalized in the following conjecture.

\begin{conjecture}\label{conj:dom_red}
Let $Q^*$ and $Q(P)$ denote the total number of jobs in stationarity in a redundancy system with $N$~servers operating according to the classical power-of-two policy and a weighted power-of-two policy with edge selection probabilities $P=(p_{\{i,j\}})_{i,j}$, respectively. Then, $Q^*$ is stochastically smaller than $Q(P)$, i.e., 
\begin{equation} \label{eq:stoch_dom_result_conj}
Q^* \le_{\mathrm{st}} Q(P).
\end{equation}
\end{conjecture}
We will establish stochastic comparison results and light-traffic limits to support the above conjecture, building on the product-form expressions for redundancy systems that will be outlined in the next subsection.

\begin{remark}\label{remark:stability}\upshape
Conjecture~\ref{conj:dom_red} implicitly assumes both $Q^*$ and $Q(P)$ to exist.  
As shown in~\cite{Gardner2016queueing} this is the case for $Q^*$ if and only if $\lambda<\mu$. It can be further deduced from~\cite{Gardner2016queueing} that the latter condition is also necessary for $Q(P)$ to exist. The sufficient condition requires the aggregate arrival rate of any subset of edges to be strictly smaller than the aggregate service rate of the servers at its endpoints, i.e., for all  $S\subseteq \mathcal{E}$ it must hold that $N\lambda \sum_{\{i,j\}\in S} p_{\{i,j\}} < \mu | \cup_{\{i,j\}\in S} \{i,j\} |$.
\end{remark}

\begin{remark}\upshape \label{rem:HT}
Note that Conjecture~\ref{conj:dom_red} contrasts with the universality result in~\cite[Theorem~1]{Cardinaels2022} in a \textit{heavy-traffic} regime.
In particular, it is shown in~\cite{Cardinaels2022} that for a broad range of weighted policies the system in heavy traffic achieves complete resource pooling and exhibits state space collapse.
\end{remark}

\subsection{Product-form expressions}
The system occupancy at time~$t$ of a system operating under the redundancy policy may be represented in terms of a vector $(c_1, \dots, c_{Q(t)})$, with $Q(t)$ denoting the total number of jobs in the system at time~$t$ and $c_q \in {\mathcal E}$. 
One can think of the occupancy vector as a central queue where $c_q \in {\mathcal E}$ indicates the type of the $q$th oldest job in the system at time $t$. It is easily verified that the system occupancy evolves as a Markov process by virtue of the exponential traffic assumptions.

From a modeling perspective, the state description yields a system that is equivalent to a system where replicas are positioned in two dedicated queues in front of the two selected servers. To illustrate this, consider the graph structure depicted in Figure~\ref{fig:example1} with $N=4$ servers and assume that the state of the system is given by $\boldsymbol{c} = (\{1,2\},\{4,1\},\{2,3\},\{4,1\})$. Hence, the oldest job in the system is replicated to servers~1 and~2, the second oldest job is replicated to servers~1 and~4, etc. Figure~\ref{fig:example2} represents the system with dedicated queues at each of the servers, while Figure~\ref{fig:example3} represents the system with a centralized queue. In the latter case a server that becomes idle scans this central queue and will initiate service of a replica of the first job it is compatible with.

It was shown in \cite{Gardner2016queueing} that, under the stability conditions mentioned in Remark~\ref{remark:stability},
the stationary distribution of the system occupancy is
\begin{equation}
\label{eq:statdistr}
\pi(c_1, \dots, c_Q) = C\prod_{i =1}^Q \frac{N\lambda p_{c_i}}{\mu(c_1, \dots, c_i)},
\end{equation}
with $C$ a normalization constant and
\begin{equation}\label{eq:mu}
\mu(c_1, \dots, c_i) = \mu |\bigcup\limits_{j = 1}^{i} \{c_j\}|
\end{equation}
the aggregate service rate of the system in the state $(c_1,\dots,c_i)$.
For instance, the stationary probability of state~$\boldsymbol{c}$ depicted in Figure~\ref{fig:example} is given by $\pi(\boldsymbol{c})= C (\frac{4\lambda}{\mu})^4\cdot\frac{p_{\{1,2\}}}{2}\frac{p_{\{4,1\}}}{3}
\frac{p_{\{2,3\}}}{4}\frac{p_{\{4,1\}}}{4}$. 

\section{Main results}
We now use the product-form expressions to assess the performance of various systems with respect to their assignment constraints.
However, the detailed job-level state description ingrained in the product-form expressions does not provide much insight into the overall performance and does not allow a direct comparison. In order to make a meaningful comparison, we therefore consider the stationary distribution of the total number of jobs in the system,~$Q$. This distribution may be expressed in terms of the detailed product-form expressions as
\begin{equation}\label{eq:summation}
\pr{Q = q} = \sum\limits_{\boldsymbol{c}\in\mathcal{E}^q} \pi(\boldsymbol{c}),
\end{equation}
with $q\ge 0$.
Hence, all $|\mathcal{E}|^q$ states $\boldsymbol{c} = (c_1,c_2,\dots,c_q)$ with $c_i \in \mathcal{E}$ for all $i=1,\dots,q$ must be aggregated to determine $\pr{Q = q}$. Besides the fact that there are exponentially many terms, the various terms are also highly different. The difference between two terms is caused by the various job types that could occur but mainly by the order in which they appear. 

However, for small systems we can enumerate all possible server rate sequences $(|c_1|,|c_1\cup c_2|,\dots,|c_1\cup\dots\cup c_q|)$, which results in stationary distributions with particular underlying structures amenable for comparison as we will show in Subsection~\ref{subsec:intro_example}. Unfortunately, this enumeration strategy does not lead to tractable expressions for larger systems. Therefore we consider larger systems in a light-traffic regime to partially suppress the complexity (captured in the normalization constant), and reveal the essential dependence of the stationary distribution on the edge selection probabilities. 
In particular, the stationary probability in~\eqref{eq:summation} reduces in a light-traffic regime to a polynomial of degree $q$ in function of the selection probabilities. This again allows to compare systems with different underlying structures as will be demonstrated in Subsection~\ref{sec:main_results}.

\subsection{Four-server systems}\label{subsec:intro_example}
We will derive closed-form expressions for the summation in~\eqref{eq:summation} for small systems, for which the computations are already quite tedious. 
The focus will be on the classical power-of-two policy and a particular subset of weighted policies, namely those policies governed by ring graphs.
Let $\epsilon\in(0,1)$ and $N$ be even, and define the edge selection probabilities as
\begin{equation}\label{eq:hetro_ring_prob}
p_{\{i,i+1\}} = \left\{
\begin{array}{lcl}
\epsilon\frac{2}{N}, & & \text{if $i$ is even}\\
(1-\epsilon)\frac{2}{N}, & & \text{if $i$ is odd}
\end{array}
\right.
\end{equation}
with $i=1,\dots N$ and $p_{\{1,N\}} = p_{\{N,N+1\}}$. Note that 
the example in Figure~\ref{fig:example1} is a special case of this setting with $N=4$ and $\epsilon = 3/4$.
When $\epsilon = 1/2$, all probabilities are equal to $N^{-1}$. The average arrival rate across all edges is given by $\lambda$ and therefore this graph is referred to as the homogeneous ring. In all other cases we refer to this underlying graph as the heterogeneous ring.

\begin{lemma}\label{lem:uniform complete graph}
The stationary distribution of the total number of jobs in a system with the uniform complete graph structure on $N=4$~servers is given by 
\begin{equation}
\mathbb{P}\{Q^*_4 = q\}  =  \frac{1}{9}\left(1-\rho\right)\left(3-\rho\right)\left(3-2\rho\right)\left\{-4 \left(\frac{2\rho}{3}\right)^q + \frac{1}{2}\left(\frac{\rho}{3}\right)^q + \frac{9}{2}\rho^q \right\},
\end{equation}
with $q\ge 0$ and $\rho \coloneqq \frac{\lambda}{\mu}<1$.
\end{lemma}
\begin{lemma}\label{lem:heterogen_ring4}
The stationary distribution of the total number of jobs in a system with the heterogeneous ring structure on $N=4$~servers is given by 
\begin{equation}\label{eq:stat_ring_4}
\begin{array}{rcl}
\mathbb{P}\{Q^{\mathrm{het}}_4(\epsilon) = q\} & = & \frac{\left(1-\rho\right)\left(1-(1-\epsilon)\rho\right)\left(1-\epsilon\rho\right)\left(3-2\rho\right)}{3-2\rho+ (1-\epsilon)\epsilon \rho^2} \cdot \\

& & \left\{
\frac{6\epsilon(1-\epsilon)}{2-9\epsilon(1-\epsilon)}\left(\frac{2\rho}{3}\right)^q + \frac{(1-\epsilon)^2}{\epsilon(2-3(1-\epsilon))}\left((1-\epsilon)\rho\right)^q  \right.\\
& & \left. + \frac{\epsilon^2}{(1-\epsilon)(2-3\epsilon)}\left(\epsilon\rho\right)^q  + \frac{1+\epsilon(1-\epsilon)}{\epsilon(1-\epsilon)}\rho^q \right\},
\end{array}
\end{equation}
with $q\ge 0$, $\epsilon \in (0,1)$ and $\rho \coloneqq \frac{\lambda}{\mu}<1$.
\end{lemma}
Note that the stationary distribution~\eqref{eq:stat_ring_4} is symmetric around $\epsilon = 1/2$, reflecting the symmetry in the edge selection probabilities. The derivations of the results in Lemmas~\ref{lem:uniform complete graph} and~\ref{lem:heterogen_ring4} are deferred to Appendix~\ref{app:coc_derivation}.

The next proposition proves a partial version of Conjecture~\ref{conj:dom_red} for systems with $N=4$ servers and weighted policies that correspond to homogeneous ring graphs.

\begin{proposition}\label{prop:stoch_dom}
Let $Q^*_4$ and $Q^{\mathrm{hom}}_4$ denote the total number of jobs in stationarity in a system with a uniform complete graph structure and a homogeneous ring, respectively, with $N=4$ servers and $\lambda<\mu$. Then, $Q^*_4$ is stochastically smaller than $Q^{\mathrm{hom}}_4$, i.e., 
\begin{equation} \label{eq:stoch_dom_result}
Q^*_4 \le_{\mathrm{st}} Q^{\mathrm{hom}}_4.
\end{equation}
\end{proposition}

The proof of Proposition~\ref{prop:stoch_dom} uses Lemmas~\ref{lem:uniform complete graph} and~\ref{lem:heterogen_ring4} and can be found in Appendix~\ref{app:stoch_dom_proof}. A numerical comparison of the above derived stationary distributions is depicted in Figure~\ref{fig:stoch_dom}. Although the figure clearly supports the result in Proposition~\ref{prop:stoch_dom}, it also suggests that the absolute differences between the various distributions are fairly small.
Furthermore, the above-described settings are compared to a setting where the ring structure is disconnected by choosing $\epsilon=0$ or $\epsilon=1$. This system is equivalent to two independent single-server queues with arrival rate~$2\lambda$ and service rate~$2\mu$. Hence, the total number of jobs in the system is determined by a sum of two independent and geometrically distributed random variables with parameter $\rho\coloneqq\lambda/\mu$, resulting in a negative binomial distribution. Alternatively, it can be seen that~\eqref{eq:stat_ring_4} indeed converges to $(q+1)(1-\rho)^2\rho^q$ when $\epsilon\downarrow 0$ or $\epsilon\uparrow 1$.

Moreover, from~\eqref{eq:stat_ring_4} it can be deduced that the probability of an empty system, i.e., $\mathbb{P}\{Q^{\mathrm{het}}_4(\epsilon) = 0\}$, decreases the more $\epsilon$ deviates from $1/2$ (Appendix~\ref{app:coc_derivation}). While considering Figure~\ref{fig:stoch_dom}, it can be observed that $\mathbb{P}\{Q^{\mathrm{het}}_4(\epsilon)\ge q\}$ increases for any fixed $q$ the more $\epsilon$ deviates from $1/2$, revealing a degradation of the system performance the more the selection probabilities of the ring structure differ from the uniform probabilities. 
 
In addition, a setting where a job can be replicated to all $N=4$~servers is considered, so $d=4$ instead of $d=2$. This fully pooled scenario is equivalent in performance to a single-server queue with arrival rate~$4\lambda$ and service rate~$4\mu$, and yields the stochastically smallest total number of jobs in the system. Indeed, the service rate in this system is always equal to~$4\mu$ when there are jobs present, while in all other cases the service rate is at most equal to~$4\mu$.

\begin{figure}[h]
\centering
\includegraphics{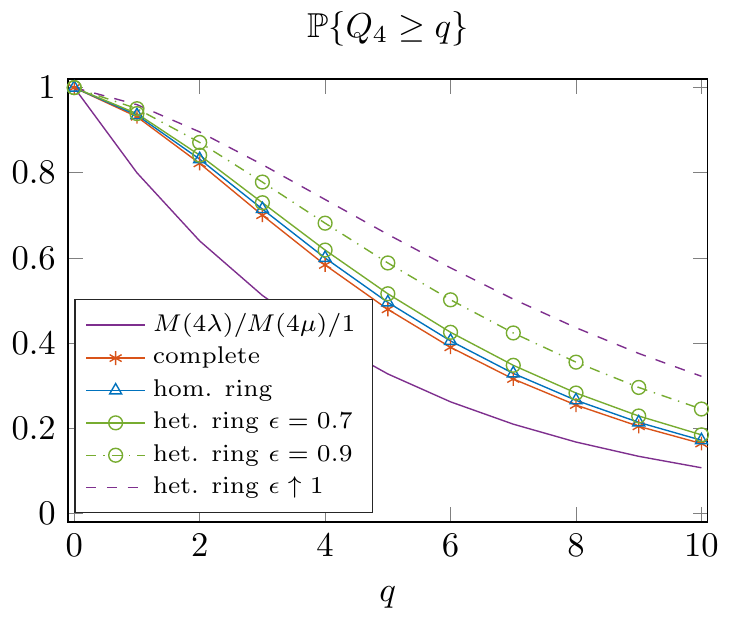}
\caption{The stationary distributions of the total number of jobs compared for various power-of-$d$ policies with $N=4$ servers and  $\rho = 0.8$: 
classical power-of-4 policy, classical power-of-2 policy, weighted power-of-2 policies governed by a homogeneous ring and heterogeneous rings with $\epsilon=0.7$, $\epsilon = 0.9$ and $\epsilon\uparrow 1$.} \label{fig:stoch_dom}
\end{figure}

The challenges with the detailed product-form expressions mentioned at the beginning of this section could be overcome by judicious state aggregation for scenarios with a small number of servers and weighted policies that correspond to ring graphs, although the expressions were already rather unwieldy in this case.
The stationary distributions (and their normalization constants) in Lemmas~\ref{lem:uniform complete graph} and~\ref{lem:heterogen_ring4} for larger values of~$N$ become more intricate, let alone their comparison for different edge selection probabilities. This makes it complicated to extend Proposition~\ref{prop:stoch_dom} for larger system sizes, though it is assumed the hold for any fixed~$N$ as illustrated by means of simulation in Appendix~\ref{app:simulation}. Alternatively, the stationary distribution of the total number of jobs in a system with $N$~servers operating under the classical power-of-$d$ policy could be derived using a similar approach as outlined in~\cite[Section~4.2.2]{gardner2017redundancy}, implicitly relying on the corresponding generating function. For any $q\ge 1$, 
\begin{equation}
\mathbb{P}\{Q^*_N = q\} 
=
C \frac{\rho^q}{\binom{N-1}{d-1}^q} \sum_{\boldsymbol{n}\in R(q)} (-1)^{|\boldsymbol{n}|+q}(|\boldsymbol{n}|)!
\prod_{j=1}^q \frac{F_j(N,d)^{n_j}}{n_j!},
\end{equation}
In the above formula, we define $
R(q) \coloneqq \{\boldsymbol{n}\in\mathbb{N}^q \colon 1\cdot n_1+2\cdot n_2+\dots +q\cdot n_q = q \},$
$|\boldsymbol{n}| \coloneqq n_1+n_2+\dots+n_q$
and $C$ as in~\cite[Theorem~2]{gardner2017redundancy}. Moreover, 
\begin{equation}
F_j(N,d) =  \sum_{\boldsymbol{k}\in\tilde{R}_j(N,d)} \prod_{i=d}^N \binom{i-1}{d-1}^{k_{i-d+1}},
\end{equation}
where $\tilde{R}_j(N,d)\subset\{0,1\}^{N-d+1}$ contains all binary vectors with precisely $j$ entries equal to~1.
Although the above expressions are presented in closed form, they are unwieldy. This would be exacerbated in case the occupancy probabilities of the weighted power-of-two policy would be derived using the explicit expression for the generating function obtained in~\cite[Proposition~1]{Cardinaels2022}, hence a general comparison between the stationary distributions seems out of reach.

Coupling arguments are commonly used as an alternative method to establish stochastic dominance properties, when the actual distributions are not tractable. However, this approach seems out of reach because the systems under consideration do not necessarily have the same number of job types or equal arrival probabilities for mutual job types. 
Also, coupling arguments would yield a stronger stochastic comparison result for the entire process over time, which may in fact not hold. This implies that a coupling approach might just be doomed to fail regardless.

\subsection{Light-traffic results}\label{sec:main_results}
As discussed in the above paragraph, stochastic dominance results in full generality do not seem within reach. As we now proceed to demonstrate however, this degree of generality can be tackled if we consider a light-traffic regime.

Let $Q_{\lambda}(P)$ be a random variable with the stationary distribution of the total number of jobs in the system with an underlying graph structure with edge selection probabilities $P=(p_{\{i,j\}})_{i,j}$. With $C$ the normalization constant in~\eqref{eq:statdistr} equal to $\mathbb{P}\{Q_{\lambda}(P) = 0 \}$ it can easily be seen that 
\begin{equation}\label{eq:general_prob}
\mathbb{P}\{Q_{\lambda}(P) = q \} = \mathbb{P}\{Q_{\lambda}(P) = 0 \}\cdot \alpha_q(P) \cdot \left(\frac{N\lambda}{\mu}\right)^q,
\end{equation}
with $q\ge 1$. Define
\begin{equation}\label{eq:alpha_k}
 \alpha_q(P) \coloneqq \sum\limits_{\boldsymbol{c}\in\mathcal{E}^q}\prod\limits_{i=1}^q\frac{p_{c_i}}{|\bigcup\limits_{j=1}^i \{c_j\}|},
\end{equation}
which only depends on the edge selection probabilities and not on the total arrival rate~$N\lambda$ or the server speed~$\mu$. Obviously, the probability of an empty system, $\mathbb{P}\{Q_{\lambda}(P) = 0 \}$, tends to one when $\lambda$ approaches zero. 
 Note that $\alpha_1(P)$ is always equal to $1/2$ since $|\{c\}| = 2$ for all $c\in\mathcal{E}$. The value $\alpha_q(P)$ for any $q\ge 2$ depends on the edge selection probabilities, for instance, $|\{c_1\}\cup\{c_2\}|$ is equal to $2$, $3$ or $4$ whenever the edges $c_1$ and $c_2$ are either the same, have one common endpoint, or no common endpoints, respectively. Therefore, $\alpha_q(P)$ will determine how various underlying structures will perform compared to each other when $\lambda$ approaches zero as formalized in the following theorem.

\begin{theorem}\label{th:light_traffic}
For any $q\ge 1$, if $P' = (p'_{\{i,j\}})_{i,j}$ and $P = (p_{\{i,j\}})_{i,j}$ are two sets of edge selection probabilities such that $\alpha_q(P')\le \alpha_q(P)$, then 
\begin{equation}
  \lim_{\lambda\downarrow 0} \frac{\mathbb{P}\{Q_{\lambda}(P')\ge q\}}{\mathbb{P}\{Q_{\lambda}(P)\ge q\}} \le 1.
\end{equation}
\end{theorem}
The following lemma allows us to establish Theorem~\ref{th:light_traffic}.

\begin{lemma}\label{lem:prelimit}
Let $P' = (p'_{\{i,j\}})_{i,j}$ and $P = (p_{\{i,j\}})_{i,j}$ be two sets of edge selection probabilities, then for any $q\ge 0$
\begin{equation}\label{eq:prelimit}
\frac{\mathbb{P}\{Q_{\lambda}(P')\ge q\}}{\mathbb{P}\{Q_{\lambda}(P)\ge q\}} = \frac{\alpha_q(P') +o(1)}{\alpha_q(P)+o(1)}
\end{equation}
as $\lambda\downarrow 0$.
\end{lemma}
\begin{proof}
The Taylor expansion of $\mathbb{P}\{Q_{\lambda}(P) = 0 \}$ near $\lambda$ equal to zero yields
\begin{equation}\label{eq:taylor}
\mathbb{P}\{Q_{\lambda}(P) = 0 \} = \sum_{k=0}^{\infty} \frac{1}{k!}\left(\frac{N\lambda}{\mu}\right)^q\dfrac{\mathrm{d}^k}{\mathrm{d}x^k} \mathbb{P}\{Q_{\lambda}(P) = 0 \}\vert_{\lambda\downarrow 0},
\end{equation}
with $x\coloneqq N\lambda/\mu$. Combining~\eqref{eq:taylor} with the observations in~\eqref{eq:general_prob} and~\eqref{eq:alpha_k} gives $\mathbb{P}\{Q_{\lambda}(P) = q \} =  \alpha_q(P)(N\lambda/\mu)^q + o(\lambda^q)$ from which~\eqref{eq:prelimit} follows.
\end{proof}
From Theorem~\ref{th:light_traffic} it can be deduced that Conjecture~\ref{conj:dom_red} holds in a light-traffic regime once we are able to establish an inequality relation for the $\alpha_q(P)$ values involved. 
More precisely, none of the weighted power-of-two policies achieves better performance than the classical power-of-two policy when it can be shown that the uniform edge selection probabilities yield the smallest values of $\alpha_q(P)$ for all $q\ge 1$. So, proving Conjecture~\ref{conj:dom_red} in a light-traffic regime boils down to an optimization problem in terms of $\alpha_q(P)$ as a function of the edge selection probabilities $P=(p_{\{i,j\}})_{i,j}$. 
Note that $\alpha_q(\cdot)$ is a multivariate polynomial of degree~$q$, hence it is continuous. Moreover, the set of edge selection probabilities is compact in $\mathbb{R}^E$, with $E=\binom{N}{2}$, implying that $\alpha_q(\cdot)$ must attain its global minimum.
With the above observations in mind, we now present the following conjecture.

\begin{conjecture}\label{conj:alpha}
Let $\alpha_q^*\coloneqq \alpha_q(P')$ with $P' = (p'_{\{i,j\}})_{i,j}$ such that $p'_{\{i,j\}}\equiv 1/\binom{N}{2}$ for all~$i$ and~$j$, $i\neq j$, then
\begin{equation}
\alpha_q^* = \min\left\{ \alpha_q(P) \mid P= \left(p_{\{i,j\}}\right)_{i,j} \right\},
\end{equation}
for all $q\ge 0$.
\end{conjecture}

The above conjecture was shared in personal communication with Brosch, Laurent and Steenkamp, who showed that $\alpha_q(P)$, as a function of $P=(p_{\{i,j\}})_{i,j}$, is a convex polynomial for $q=2$ and 3. Hence, choosing all edge selection probabilities to be uniform will minimize $\alpha_q(\cdot)$~\cite[Theorem~2]{Brosch2020}. Polak later extended this convexity result to $q\le 9$ \cite[Theorem 1.1]{Polak2021}. 
In~\cite{Polak2021} convexity is established once the Hessian matrix of $\alpha_q$ is positive semidefinite via a symmetry reduction. Proving semidefiniteness of the obtained lower-dimensional matrices increases in complexity as it becomes computationally harder to obtain the matrix coefficients for larger values of $q$. Combining \cite[Theorem 1.1]{Polak2021} with Theorem~\ref{th:light_traffic} results in the following corollary.
\begin{corollary}
Let $Q_{\lambda}^*$ and $Q_{\lambda}(P)$ denote the total number of jobs in stationarity in a system with $N$~servers operating according to the classical power-of-two policy and a weighted power-of-two policy with edge selection probabilities $P=(p_{\{i,j\}})_{i,j}$, respectively. Then, for $q\le 9$,
\begin{equation} \label{eq:corr_LT}
\lim_{\lambda\downarrow 0}  \frac{\mathbb{P}\{Q_{\lambda}^*\ge q\}}{\mathbb{P}\{Q_{\lambda}(P)\ge q\}} \le 1.
\end{equation}
\end{corollary}

Table~\ref{tab:alpha_fraction} gives a comparison between $\alpha_q^*$ and $\alpha_q(P)$ for several values of $q$ and when the underlying structure is a ring. 
We observe that for a fixed number of servers~$N$, in analogy to the observations in Subsection~\ref{subsec:intro_example},
the performance of the system governed by the homogeneous ring is closer to the performance of the uniform case than the heterogeneous rings
as $\alpha_q^*/\alpha_q(P)$ in this case is closer to 1. 

\begin{table}[h]
\centering
\begin{tabular}{l|cccc|ccc}
                           & \multicolumn{4}{c|}{$N=4$}  & \multicolumn{3}{c}{$N=8$} \\ \cline{2-8} 
                           & $q=2$   & $q=4$   & $q=10$ & $q=16$ & $q=2$   & $q=4$  & $q=10$ \\ \hline
hom. ring                  & 0.9804  & 0.9432  & 0.9046 & 0.9004 & 0.9754  & 0.8947 &      0.6586  \\
het. ring $\epsilon = 0.7$ & 0.9713  & 0.9055  & 0.8957 & 0.7850 & 0.9700  & 0.8707 &    0.5831   \\
het. ring $\epsilon =0.9$  & 0.9448  & 0.8063  & 0.5481 & 0.4509 & 0.9543  & 0.8051 &     0.4095  
\end{tabular}
\caption{The fraction of $\alpha^*_q/\alpha_q(P)$ for various values of $q$ when the edge selection probabilities~$P$ correspond to a ring structure. Lemma~3 implies that this fraction corresponds to $\mathbb{P}\{Q^*_{\lambda}\ge q\}/\mathbb{P}\{Q_{\lambda}(P)\ge q\}$ when $\lambda\downarrow 0$.}\label{tab:alpha_fraction}
\end{table}

Computing $\alpha_q$ for a given set of edge selection probabilities $P = (p_{\{i,j\}})_{i,j}$ is time-consuming as it requires summation over $|\mathcal{E}|^q$ terms, and also formed the bottleneck to prove the convexity results in~\cite{Polak2021} for values of $q\ge 10$. However, the above numerical results support the statement in Conjecture~\ref{conj:alpha}.

\begin{remark}\label{remark:BD}\upshape
The condition $\alpha_q(P') \le \alpha_q(P)$ in Theorem~\ref{th:light_traffic} is not sufficient to establish the stochastic dominance result in Conjecture~\ref{conj:dom_red} for any fixed value of~$\lambda$, even when this inequality  could be shown to hold for all $q\ge 1$, which is due to the behavior of the normalization constant. However, a sufficient condition would be
\begin{equation}\label{eq:condition_non_LT}
 \frac{\alpha_{q-1}(P')}{\alpha_{q}(P')} \ge \frac{\alpha_{q-1}(P)}{\alpha_{q}(P)}
\end{equation}
for all $q\ge 1$, which also implies the condition in Theorem~\ref{th:light_traffic}. The fact that $Q(P')\le_{\text{st}} Q(P)$ once condition~\eqref{eq:condition_non_LT} is fulfilled for all $q\ge 1$ follows from a direct comparison of the respective stationary distributions in~\eqref{eq:general_prob}. Details of the proof are deferred to Appendix~\ref{app:BD_dominance}.
\end{remark}

\begin{remark} \upshape
We used the product-form distributions to establish the
light-traffic result in Theorem~\ref{th:light_traffic}, while usually a light-traffic approach is only considered when explicit formulas are lacking, and then based on the powerful framework developed by Reiman and Simon~\cite{Reiman1989}.
The latter framework outlines an approach to determine the coefficients of the Taylor expansion in~\eqref{eq:taylor}.
To derive these coefficients, one has to take into account the arrival and departure times of individual jobs, as well as the exact service rate at each moment in time, which is complicated by the fact that multiple servers can be processing a replica of the same job. 
Hence, it is notationally and computationally more convenient to leverage the product-form expressions which directly furnish the desired coefficients in terms of~\eqref{eq:alpha_k}.
\end{remark}

\section{Discussion} \label{sec:extensions}

\subsection{Design implications}
In Section~\ref{sec:main_results} we proved a partial version of Conjecture~\ref{conj:dom_red} implying that non-uniform edge selection probabilities cannot yield better performance than uniform ones in a light-traffic regime. In many situations however, strictly uniform edge selection probabilities may simply not be feasible because of assignment constraints.  Theorem~\ref{th:light_traffic}, in conjunction with Lemma~\ref{lem:prelimit}, then provides a specific guideline for the design of an efficient assignment policy subject to these constraints as we will now illustrate.

Assume that there are $K$ different job types with arrival rates $\lambda_1,\dots, \lambda_K$ and $\sum_{k=1}^K \lambda_k = N\lambda$. In the system design one has to choose (once and for all) for each job type $k=1,\dots,K$ which server pair (or edge) $e\in\mathcal{E}$ is eligible for assignment.
For example, the various job types may correspond to requests for different data objects, which each can only be stored at two servers.
The assignment policy can thus be represented in terms of binary decision variables $(x_{e,k})_{e,k}$, which are equal to 1 if job type~$k$ can be assigned to the servers at the endpoints of edge~$e$, and 0 otherwise. The aim is to find values for the variables $\boldsymbol{x} = (x_{e,k})_{e,k}$ yielding edge selection probabilities $P(\boldsymbol{x}) = (p_e(\boldsymbol{x}))_e$ that stochastically minimize the number of jobs in the system. The edge selection probabilities may be expressed in terms of the decision variables as
\begin{equation}\label{eq:edge_sel_prob_optimization}
p_{e}(\boldsymbol{x}) = \frac{1}{N\lambda}\sum\limits_{k=1}^K \lambda_k x_{e,k}~~~\text{for all~}e\in\mathcal{E}.
\end{equation}

Recalling that $\mathbb{P}\{Q(P(\boldsymbol{x})) \ge q \}  = \alpha_q(P(\boldsymbol{x}))\cdot(N\lambda/\mu)^q + o(\lambda^q)$ for $q\ge 1$ and $\alpha_1(P(\boldsymbol{x}))\equiv 1/2$, Lemma~\ref{lem:prelimit} and Theorem~\ref{th:light_traffic} suggest that the following minimization problem should be solved in order to obtain the ideal distribution of the various types as captured by~\eqref{eq:edge_sel_prob_optimization}:
\begin{subequations}
\begin{align}
& \text{min} && \alpha_2\left( P(\boldsymbol{x})\right)  = \sum\limits_{\boldsymbol{c}\in\mathcal{E}^2}  \frac{p_{c_1}(\boldsymbol{x})}{2}\frac{p_{c_2}(\boldsymbol{x})}{|\{c_1\}\cup\{c_2\}|} &&&~ \\
&\text{s.t.} && \sum\limits_{e\in\mathcal{E}} x_{e,k} =1 && \text{for all~} k=1,\dots,K, \label{eq:constraint1} \\
&~ && N\lambda\sum\limits_{e\in\mathcal{I}}p_e(\boldsymbol{x}) =  \sum\limits_{k=1}^K\lambda_k\sum\limits_{e\in\mathcal{I}}x_{e,k} < \mu(\mathcal{I}) &&  \text{for all~}\mathcal{I}\subseteq\mathcal{E}, \label{eq:constraint2}\\
 &~ && x_{e,k} \in \{0,1\} && \text{for all~} (e,k)\in |\mathcal{E}|\times K. \label{eq:constraint3}
\end{align}
\end{subequations}
In the above optimization problem, \eqref{eq:constraint1} guarantees that each job type~$k$ gets assigned to precisely one server pair or edge. Moreover, \eqref{eq:constraint2} ensures that the system is stable, with with $\mu(\mathcal{I})$ as defined in~\eqref{eq:mu}, in accordance with the stability conditions identified in~\cite{Gardner2016queueing}.

\begin{remark}\upshape
A feasible solution $\boldsymbol{x} = (x_{e,k})_{e,k}\in \{0,1\}^{|\mathcal{E}|\times K}$ cannot be constructed for all arrival rate vectors $\{\lambda_k\}_k$. This is for instance the case when the sufficient conditions to guarantee stability, stated in Remark~\ref{remark:stability}, are not satisfied, even if the necessary condition $N\lambda = \sum_{k=1}^K \lambda_k < N\mu$  is met. A simple counter example can be constructed in a system with $N=4$ servers and $K=2$ job types with arrival rates $\lambda_1 = (2+\delta)\mu$ and $\lambda_2 = (2-2\delta)\mu$ for any $\delta\in(0,1)$. Even though $\lambda <\mu$, there exists no allocation of the two job types that yields a stable system as $\lambda_1 \ge 2\mu$.

\end{remark}


\subsection{Other load balancing policies}
The broader theme of the present paper is comparing the performance of weighted power-of-$d$ policies with that of the classical power-of-$d$ policy. 
As mentioned earlier, the notion that the performance of the latter policy serves as an upper bound for the performance of the former policies was supported through heuristic arguments and simulations by Gast~\cite{gast2015power}, Mitzenmacher~\cite{Mitzenmacher_thesis} and Turner~\cite{Turner1998} in a JSQ context.
We proved that this property indeed holds for redundancy policies both in small systems and in systems of arbitrary size in the light-traffic regime. We focused on redundancy policies in view of the explicit product-form distributions, but we expect that the stochastic comparison results extend to load balancing policies beyond redundancy policies. This introduces interesting directions for further research as the above methods cannot directly be applied to analyze these alternative policies.

The redundancy policy described in Section~\ref{sec:model_description} is often referred to as the \textit{redundancy cancel-on-completion} (c.o.c.)\ policy. 
A natural policy to investigate as well is the \emph{redundancy cancel-on-start} (c.o.s.) policy. Instead of discarding the redundant replicas once one of them finished service, redundant replicas are now discarded once one of them starts service. 
Hence, the job is served at the server where its replica encountered the smallest workload, yielding an alternative implementation of the \textit{Join-the-Smallest-Workload} (JSW) policy~\cite{Adan2018,ayesta2018unifying}.
Whenever several replicas find idle servers upon arrival, the job will undergo service at the server that has been idle for the longest time, referred to as Assign-to-Longest-Idle-Server (ALIS). 

The system occupancy at time $t$ is given by $(c_1,\dots,c_{Q'(t)}; u_1,\dots,u_{L(t)})$ with $Q'(t)$ the total number of \textit{waiting} jobs in the system at time $t$ and $c_q\in\mathcal{E}$ denoting the $q$th oldest waiting job in the system. There are $L(t)$ idle servers at time $t$ and server $u_l\in\{1,\dots,N\}$ is the $l$th longest idle server in the system. Note that it is not feasible to have simultaneously waiting type-$\{i,j\}$ jobs and either server $i$ or $j$ idle.
It was shown in \cite{Adan2018,ayesta2018unifying} that, under suitable stability conditions, the stationary distribution of the system occupancy is
\begin{equation}\label{eq:stat_distr_cos}
\pi^{\text{c.o.s.}}(c_1,\dots,c_{Q'}; u_1,\dots,u_{L}) = C'\prod\limits_{i=1}^{Q'} \frac{N\lambda p_{c_{i}}}{\mu(c_1,\dots,c_i)} \prod\limits_{l=1}^{L}\frac{\mu}{\lambda_{\mathcal{C}(u_1,\dots,u_l)}},
\end{equation}
with $C'$ the normalization constant, $\mu(c_1,\dots,c_i)$ as defined in~\eqref{eq:mu} and 
\begin{equation}
\lambda_{\mathcal{C}(u_1,\dots,u_l)} = N\lambda \sum\limits_{e\in\mathcal{E}\colon e\cap\{u_1,\dots,u_l\} \neq \emptyset} p_{e}
\end{equation}
the total arrival rate of jobs that can be served by the idle servers $\{u_1,\dots,u_l\}$. Comparing this product-form expression with~\eqref{eq:statdistr} for the redundancy c.o.c.\ policy reveals that obtaining stationary probabilities at an aggregate level is now also affected by the servers that are idle and the relative times they became idle.

For small systems it is possible to obtain the stationary distribution of the total number of jobs in the system from~\eqref{eq:stat_distr_cos}. The next two lemmas mirror the results in Subsection~\ref{subsec:intro_example}.

\begin{lemma}\label{lem:cos_complete}
The stationary distribution of the total number of jobs in a system with the uniform complete graph structure on $N=4$~servers operating under the redundancy c.o.s.\ policy is given by 
\begin{equation}\label{eq:stat_complete_cos}
\mathbb{P}\{Q^{\text{\emph{*,c.o.s.}}}_4 = q \} = C^{\text{\emph{*,c.o.s.}}}_4\left\{ 20\rho^{q} + \frac{4}{3^3}\left(\frac{\rho}{3}\right)^{q}-\frac{5\cdot 2^5}{3^3} \left(\frac{2\rho}{3}\right)^{q}\right\},
\end{equation}
with $q\ge 1$, $\rho \coloneqq \frac{\lambda}{\mu}<1$ and
\begin{equation}
C^{\text{\emph{*,c.o.s.}}}_4 = \dfrac{\left(1-\rho\right)\left(3-\rho\right)\left(3-2\rho\right)}{\left(1+\rho\right)\left(3+\rho\right)\left(3+2\rho\right)}.
\end{equation}
\end{lemma}

\begin{lemma}\label{lem:cos_ring}
The stationary distribution of the total number of jobs in a system with the homogeneous ring structure on $N=4$~servers operating under the redundancy c.o.s.\ policy is given by 
\begin{equation}\label{eq:stat_ring_cos}
\mathbb{P}\{Q^{\text{\emph{hom,c.o.s.}}}_4 = q \} = C^{\text{\emph{hom,c.o.s.}}}_4\left\{ 5\rho^{q} + \frac{1}{3}\left(\frac{\rho}{2}\right)^{q}-2 \left(\frac{2\rho}{3}\right)^{q}\right\},
\end{equation}
with $q\ge 1$, $\rho \coloneqq \frac{\lambda}{\mu}<1$ and
\begin{equation}
C^{\text{\emph{hom,c.o.s.}}}_4 = \dfrac{48\left(1-\rho\right)\left(2-\rho\right)\left(3-2\rho\right)}{-2\rho^3+55\rho^2+121\rho+66}.
\end{equation}
\end{lemma}
The proofs of Lemmas~\ref{lem:cos_complete} and~\ref{lem:cos_ring} are given in Appendix~\ref{app:cos}.
A comparison between the two stationary distributions in Lemmas~\ref{lem:cos_complete} and~\ref{lem:cos_ring} for various values of $\rho$ can be found in Figure~\ref{fig:dominance_cos}, which suggests that an equivalent result as in Proposition~\ref{prop:stoch_dom} holds for the redundancy c.o.s.\ policy, namely, $Q_4^{*,\text{c.o.s.}}\le_{\text{st}} Q_4^{\text{hom,c.o.s.}}$.

\begin{figure}[h]
\begin{subfigure}[t]{0.49\textwidth}
  \centering
  \includegraphics{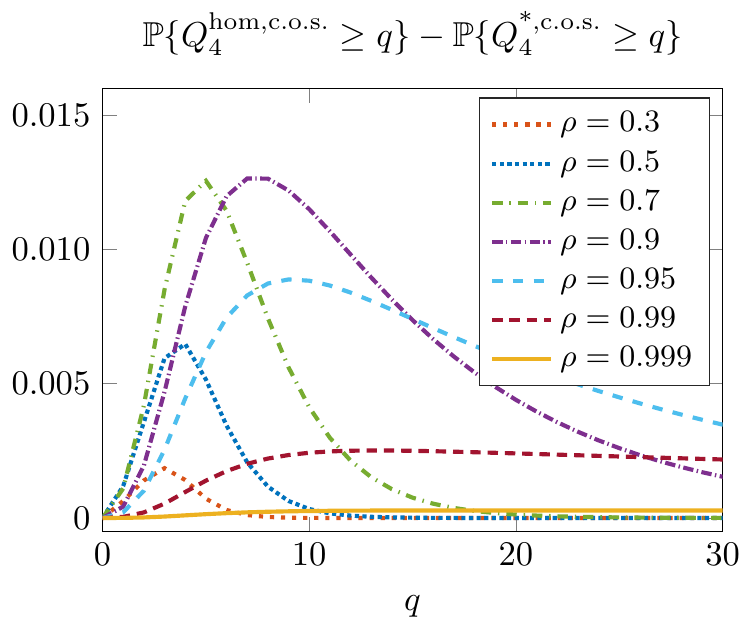}
  \caption{Redundancy c.o.s., based on Lemmas~\ref{lem:cos_complete} and~\ref{lem:cos_ring}.}
  \label{fig:dominance_cos}
\end{subfigure}%
~
\begin{subfigure}[t]{.49\textwidth}
  \centering
  \includegraphics{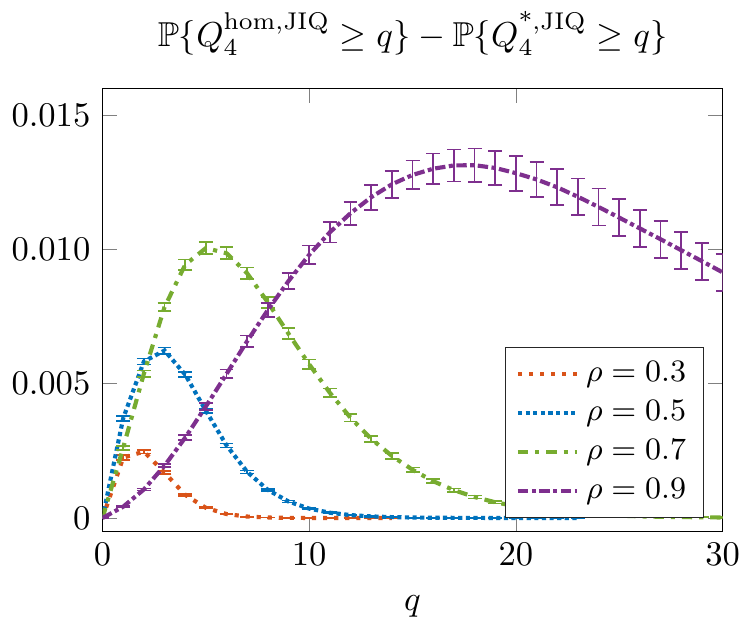}
  \vspace{-0.45cm}
  \caption{JIQ, based on simulations.}
  \label{fig:dominance_JIQ}
\end{subfigure}
\caption{A comparison between the stationary distributions of the total number of jobs for four-server systems with a uniform complete graph and a homogeneous ring as underlying structures.}
\label{fig:dominance}
\end{figure}

\begin{remark}\upshape
In Figure~\ref{fig:dominance_cos} it can be seen that the difference between the cumulative distributions of $Q^{\text{hom,c.o.s.}}_4$ and $Q^{\text{*,c.o.s.}}_4$, though still positive, narrows for values of $\lambda$ approaching~$\mu$. This observation is in line with the heavy-traffic results in~\cite[Theorem~1]{Cardinaels2022} for both redundancy c.o.c.\ and c.o.s.\ policies, showing that $(1-\lambda/\mu)Q(P)$ converges in distribution to an exponentially distributed random variable with unit mean for any set of edge selection probabilities that do not create local bottlenecks when $\lambda\uparrow\mu$.
\end{remark}

A crucial difference between the product-form distributions for redundancy c.o.c.\ and c.o.s.\ is the fact that the normalization constant of the former corresponds to the probability that the system is completely idle, while for the latter it corresponds to the probability that there are no waiting jobs in the system and all servers are occupied. From this it immediately follows that the normalization constant will not tend to 1 in a light-traffic regime, implying that a direct generalization of the reasoning in Subsection~\ref{sec:main_results} is not applicable. It is worthwhile to note that there exist alternative state descriptors for which the normalization constant does coincide with the probability that the system is completely idle, see for instance~\cite[Theorem 3.10]{Gardner2020}. However, the corresponding product-form stationary distribution is inherently more complex than the one in~\eqref{eq:stat_distr_cos},
which would yield additional challenges when proving an equivalent formulation of Theorem~\ref{th:light_traffic} for the redundancy c.o.s.\ policy.

As mentioned earlier, the notion that non-uniform sampling cannot yield better performance than uniform sampling is a quite natural one and expected to apply more broadly for load balancing policies beyond JSQ and redundancy strategies.

We will now numerically illustrate this for the so-called Join-the-Idle-Queue (JIQ) policy, which has attracted significant attention in the load balancing literature recently. The JIQ policy assigns an arriving job to an idle (compatible) server, if any. Otherwise, the job is assigned to a randomly selected (compatible) server. 
Since no expressions are available for the stationary distribution, we used simulations to compare the empirical distributions of systems with a homogeneous ring and a uniform complete graph for $N=4$ servers for various values of~$\rho$. From Figure~\ref{fig:dominance_JIQ} it can again be observed that the stochastic ordering result holds, i.e., $Q_4^{*,\text{JIQ}}\le_{\text{st}} Q_4^{\text{hom,\text{ JIQ}}}$. The comparison in Figure~\ref{fig:dominance_JIQ} is based on 50 simulation runs per value of~$\rho$, each run consisting of 10 000 000 events. 
Besides the average difference between the cumulative distributions of $Q^{\text{hom,JIQ}}_4$ and $Q^{\text{*,JIQ}}_4$, also its 95$\%$ confidence intervals are plotted.

\bibliographystyle{abbrv}
\bibliography{references}

\begin{appendices}
\numberwithin{equation}{section}

\section{Four-server systems}\label{app:stoch_dom}
\subsection{Proofs of Lemmas~\ref{lem:uniform complete graph} and~\ref{lem:heterogen_ring4}}\label{app:coc_derivation}
\begin{proof}[Proof of Lemma~\ref{lem:uniform complete graph}]
First of all, note that 
\begin{equation}
\mathbb{P}\{Q^*_4 = q\} = C\sum_{\boldsymbol{c}\in \mathcal{E}^q} \prod\limits_{i=1}^q\frac{N\lambda p_{c_i}}{\mu(c_1,\dots,c_i)} = C\left(\frac{2\lambda}{3\mu}\right)^q \sum_{\boldsymbol{c}\in \mathcal{E}^q} \prod\limits_{i=1}^q \left( |c_1\cup\dots\cup c_i|\right)^{-1},
\end{equation}
 because $p_{\{i,j\}} =1/6$ for each edge in the setting under consideration and $|c_1\cup\dots\cup c_i|$ denotes the total number of busy servers processing jobs of types $c_1,\dots, c_i$. 
 
 Then, we can enumerate all server rate sequences $(|c_1|,|c_1\cup c_2|,\dots,|c_1\cup\dots\cup c_q|)$ and count how many states will result in these particular sequences. 
 First, the sequence $(2,\dots,2)$ can only be generated by states where all the jobs are of the same type, hence there are six such sequences. Second, the sequences $(2,\dots,2,3,\dots,3)$ with $q_1$ 3 entries are considered. For the first job in the state there are six options, and all following $q-q_1-1$ jobs must be of the same type as the first job, say $\{i,j\}$. An arrival of the $(q-q_1+1)$th job increases the number of covered servers from two to three, and there are four such job types that can do this. For example, say edge $\{j,k\}$ is selected. Then, all remaining jobs can be of types $\{i,j\}$, $\{j,k\}$ or $\{i,k\}$. This results in $6\cdot 4\cdot 3^{q_1-1}$ different states that lead to this server rate sequence. Third, in a similar way one can see that there are $6^{q_1}$ states that will lead to the server rate sequence $(2,\dots,2,4,\dots,4)$ with $q_1$ 4 entries. The fourth server rate sequence is given by
\[
(\underbrace{2,\dots,2}_{q-q_1-q_2},\underbrace{3,\dots,3}_{q_1},\underbrace{4,\dots,4}_{q_2}).
\]
In total there are $6\cdot 4 \cdot 3^{q_1-1}\cdot3 \cdot 6^{q_2-1} = 4\cdot 3^{q_1}\cdot 6^{q_2}$ states that result in the above service rate vector.

 Summing over all possible values of $q_1$ and $q_2$ results in 
 \begin{equation}\label{eq:stat_unif_inter}
 \begin{array}{rcl}
 \mathbb{P}\{Q^*_4 = q\} & = &  C\left(\frac{2\lambda}{3\mu}\right)^q \left\{  \frac{6}{2^q} {+} 8\sum\limits_{q_1 = 1}^{q-1} \dfrac{3^{q_1}}{2^{q-q_1}3^{q_1}} {+} \sum\limits_{q_1 = 1}^{q-1} \dfrac{6^{q_1}}{2^{q-q_1}4^{q_1}} {+} 4\sum\limits_{q_1 = 1}^{q-2} \sum\limits_{q_2 = 1}^{q-q_1-1} \dfrac{ 3^{q_1}\cdot 6^{q_2}}{2^{q-q_1-q_2}3^{q_1}4^{q_2}} \right\}\\

  & =&  C\left\{-4 \left(\frac{2\lambda}{3\mu}\right)^q + \frac{1}{2}\left(\frac{\lambda}{3\mu}\right)^q + \frac{9}{2}\left(\frac{\lambda}{\mu}\right)^q \right\}.
    \end{array}
 \end{equation}
 
 Finally, the expression for the normalization constant $C$ can be obtained by summing the expression~\eqref{eq:stat_unif_inter} for all values $q\ge0$ and this leads to
 \begin{equation}
 C = \frac{1}{9}\left(1-\rho\right)\left(3-\rho\right)\left(3-2\rho\right),
 \end{equation}
with $\rho\coloneqq \frac{\lambda}{\mu}$. This concludes the proof.
\end{proof}

\begin{proof}[Proof of Lemma~\ref{lem:heterogen_ring4}]
The method to derive the stationary distribution when the underlying graph is a ring graph is similar to the method outlined in the proof of Lemma~\ref{lem:uniform complete graph}. The main difference is that the edge selection probabilities are no longer uniform, but given by the probabilities in~\eqref{eq:hetro_ring_prob}. The reasoning from the previous proof can be extended by taking into account whether the edge $\{i,i+1\}$ of the first job is such that $i$ is even or odd. This will lead to
\begin{equation}
\begin{array}{rcl}
\mathbb{P}\{Q^{\mathrm{het}}_4(\epsilon) = q\} & = &  C\left(\frac{4\lambda}{\mu}\right)^q  \sum\limits_{(c_1,\dots,c_q)\in \mathcal{E}^q} \prod\limits_{i=1}^q \frac{p_{c_i}}{|c_1\cup\dots\cup c_i|}\\
& = & C\left(\frac{4\lambda}{\mu}\right)^q  \left\{ 
\dfrac{1}{2^q}2\left(\frac{\epsilon}{2}\right)^q + \dfrac{1}{2^q}2\left(\frac{1-\epsilon}{2}\right)^q  \right. \\

& &\left. 
+ \sum\limits_{q_1 = 1}^{q-1} \dfrac{8}{2^{q-q_1}3^{q_1}}\left[ \frac{1-\epsilon}{2}\left(\frac{\epsilon}{2}\right)^{q-q_1}\left(\frac{1}{2}\right)^{q_1} + \frac{\epsilon}{2}\left(\frac{1-\epsilon}{2}\right)^{q-q_1}\left(\frac{1}{2}\right)^{q_1}\right]\right.\\

& &\left. 
+ \sum\limits_{q_1 = 1}^{q-1} \dfrac{2}{2^{q-q_1}4^{q_1}}\left[ \left(\frac{\epsilon}{2}\right)^{q-q_1+1} + \left(\frac{1-\epsilon}{2}\right)^{q-q_1+1}\right] \right. \\

& &\left. 
+ \sum\limits_{q_1 = 1}^{q-2} \sum\limits_{q_2 = 1}^{q-q_1-1} \dfrac{4}{2^{q-q_1}3^{q_1}4^{q_2}}\left[ \left(\frac{\epsilon}{2}\right)^{q-q_1-q_2}\frac{1-\epsilon}{2}\left(\frac{1}{2}\right)^{q_1} + \left(\frac{1-\epsilon}{2}\right)^{q-q_1-q_2}\frac{\epsilon}{2}\left(\frac{1}{2}\right)^{q_1}\right]
 \right\}.
\end{array}
\end{equation}
Simplification of the previous expression and summing over all states to determine the normalization constant $C$ will eventually lead to the expression~\eqref{eq:stat_ring_4}. This concludes the proof.
\end{proof}

It can be seen that the probability to observe an empty system, i.e.,
\begin{equation}
\mathbb{P}\{Q^{\mathrm{het}}_4(\epsilon) = 0\} = \frac{\left(1-\rho\right)\left(1-(1-\epsilon)\rho\right)\left(1-\epsilon\rho\right)\left(3-2\rho\right)}{3-2\rho+ (1-\epsilon)\epsilon \rho^2} \in [0,1],
\end{equation}
is maximal when $\epsilon = 1/2$ and decreases the more $\epsilon$ deviates from $1/2$. This observation can be made after inspection of the derivative
\begin{equation}
\frac{\mathrm{d}}{\mathrm{d}\epsilon} \mathbb{P}\{Q^{\mathrm{het}}_4(\epsilon) = 0\} = 
\rho \left(1-\rho\right)\left(3-2\rho\right)\left(2-\rho\right)\frac{1-2\epsilon}{\left(3-2\rho+(1-\epsilon)\epsilon\rho^2\right)^2},
\end{equation}
where the only non-positive factor is $(1-2\epsilon)$.

Note that for the homogeneous ring on $N=4$ servers, the queue length stationary distribution is given by
\begin{equation}
\mathbb{P}\{Q^{\mathrm{hom}}_4 = q\} = \frac{\left(1-\rho\right)\left(2-\rho\right)\left(3-2\rho\right)}{\left(6-\rho\right)}
\left\{ -6 \left(\frac{2\rho}{3}\right)^q + 2\left(\frac{\rho}{2}\right)^q + 5\rho^q  \right\}
\end{equation}
via substitution of $\epsilon=1/2$ into~\eqref{eq:stat_ring_4}.

\subsection{Proof of Proposition~\ref{prop:stoch_dom}}\label{app:stoch_dom_proof}
\begin{proof}[Proof of Proposition~\ref{prop:stoch_dom}]
Using the stationary distributions derived in Lemmas~\ref{lem:uniform complete graph} and \ref{lem:heterogen_ring4} (with $\epsilon = 1/2$) it can be seen that for $q\ge 0$: 
\begin{equation}
\begin{array}{rcl}
\mathbb{P}\{Q^*_4 \ge q\}  &=&  -\frac{4}{3}\left(1-\frac{\lambda}{\mu}\right)\left(3-\frac{\lambda}{\mu}\right)\left(\frac{2\lambda}{3\mu}\right)^q + \frac{1}{6} \left(1-\frac{\lambda}{\mu}\right)\left(3-2\frac{\lambda}{\mu}\right)\left(\frac{\lambda}{3\mu}\right)^q \\
& & + \frac{1}{2}\left(3-\frac{\lambda}{\mu}\right)\left(3-2\frac{\lambda}{\mu}\right)\left(\frac{\lambda}{\mu}\right)^q\\

\mathbb{P}\{Q^{\mathrm{hom}}_4 \ge q\}  &=&  -18\frac{\left(1-\frac{\lambda}{\mu}\right)\left(2-\frac{\lambda}{\mu}\right)}{\left(6-\frac{\lambda}{\mu}\right)}\left(\frac{2\lambda}{3\mu}\right)^q + 4 \frac{\left(1-\frac{\lambda}{\mu}\right)\left(3-2\frac{\lambda}{\mu}\right)}{\left(6-\frac{\lambda}{\mu}\right)}\left(\frac{\lambda}{2\mu}\right)^q \\
& & + 5\frac{\left(3-2\frac{\lambda}{\mu}\right)\left(2-\frac{\lambda}{\mu}\right)}{\left(6-\frac{\lambda}{\mu}\right)}\left(\frac{\lambda}{\mu}\right)^q.\\
\end{array}
\end{equation}
It is sufficient to show that $\mathbb{P}\{Q^{\mathrm{hom}}_4 \ge q\} - \mathbb{P}\{Q^*_4 \ge q\} > 0$ for any $q\ge 1$ such that~\eqref{eq:stoch_dom_result} holds. Dividing by common factors results in the following equivalent inequality:
\begin{equation}
\begin{array}{rcl}
g_{\lambda}(q) &:=&  \frac{9}{2} \left(1-\frac{\lambda}{\mu}\right)\left(2+\frac{\lambda}{\mu}\right)\left(3-2\frac{\lambda}{\mu}\right) + 36 \left(1-\frac{\lambda}{\mu}\right)\left(3-2\frac{\lambda}{\mu}\right)\left(\frac{1}{2}\right)^q \\
& & - \frac{3}{2}\left(1-\frac{\lambda}{\mu}\right)\left(3-2\frac{\lambda}{\mu}\right)\left(6-\frac{\lambda}{\mu}\right)\left(\frac{1}{3}\right)^q  -6 \left(1-\frac{\lambda}{\mu}\right)\left(3-2\frac{\lambda}{\mu}\right)\left(6+\frac{\lambda}{\mu}\right)\left(\frac{2}{3}\right)^q>0.
\end{array}
\end{equation}
Now, it is clear that $g_{\lambda}(1) = \frac{\lambda}{\mu}\left(1-\frac{\lambda}{\mu}\right)\left(3-2\frac{\lambda}{\mu}\right)> 0$. Moreover, it can be seen that 
\begin{equation}
\begin{array}{rcl}
g_{\lambda}(q+1)-g_{\lambda}(q) &=& \left(1-\frac{\lambda}{\mu}\right)\left(3-2\frac{\lambda}{\mu}\right) \left\{ -18 \left(\frac{1}{2}\right)^q+\left(6-\frac{\lambda}{\mu}\right)\left(\frac{1}{3}\right)^q + 2\left(6+\frac{\lambda}{\mu}\right)\left(\frac{2}{3}\right)^q \right\}.
\end{array}
\end{equation}
If $q=1$, then
\begin{equation}
\begin{array}{rcl}
g_{\lambda}(2)-g_{\lambda}(1) &=& \left(1-\frac{\lambda}{\mu}\right)\left(3-2\frac{\lambda}{\mu}\right) \frac{2}{3}>0.
\end{array}
\end{equation}
If $q\ge2$, the difference can be lower bounded as follows
\begin{equation}
\begin{array}{rcl}
g_{\lambda}(q+1)-g_{\lambda}(q) &>& \left(1-\frac{\lambda}{\mu}\right)\left(3-2\frac{\lambda}{\mu}\right) \left\{ -18 \left(\frac{1}{2}\right)^q+5\left(\frac{1}{3}\right)^q + 12\left(\frac{2}{3}\right)^q \right\}\\

&\ge& 6\left(1-\frac{\lambda}{\mu}\right)\left(3-2\frac{\lambda}{\mu}\right) \left\{ -3 \left(\frac{1}{2}\right)^q + 2\left(\frac{2}{3}\right)^q \right\}>0.
\end{array}
\end{equation}
This concludes the proof. 
\end{proof}

\section{Remark~\ref{remark:BD}}\label{app:BD_dominance}
In this section we will elaborate on the statement made in Remark~\ref{remark:BD}.

\begin{theorem}\label{th:BD_extension}
If for any $q\ge 1$, $P' = (p'_{\{i,j\}})_{i,j}$ and $P = (p_{\{i,j\}})_{i,j}$ are two sets of edge selection probabilities such that 
\begin{equation}\label{eq:stoch_dom_BD_condition}
 \frac{\alpha_{q-1}(P')}{\alpha_{q}(P')} \ge \frac{\alpha_{q-1}(P)}{\alpha_{q}(P)},
\end{equation}
then $Q(P')$ is stochastically smaller than $Q(P)$, i.e., 
\begin{equation}\label{eq:stoch_dom_BD}
Q(P') \le_{\mathrm{st}} Q(P).
\end{equation}
\end{theorem}

\begin{proof}
We will show that $\mathbb{P}\{ Q(P)\ge q \} - \mathbb{P}\{Q(P') \ge q\} \ge 0$ for any $q\ge 1$. Relying on~\eqref{eq:general_prob}, this is equivalent to
\begin{equation}
\dfrac{\sum\limits_{k=q}^{\infty}\alpha_k(P)\left(\frac{N\lambda}{\mu}\right)^k}{\sum\limits_{k=0}^{\infty}\alpha_k(P)\left(\frac{N\lambda}{\mu}\right)^k}
- \dfrac{\sum\limits_{l=q}^{\infty}\alpha_l(P')\left(\frac{N\lambda}{\mu}\right)^l}{\sum\limits_{l=0}^{\infty}\alpha_l(P')\left(\frac{N\lambda}{\mu}\right)^l} \ge 0.
\end{equation}
The denominator of the resulting fraction is always non negative, simplification of the numerator yields the following equivalent expression,
\begin{equation}
\sum\limits_{k=q}^{\infty}\sum\limits_{l=0}^{q-1} \left(\frac{N\lambda}{\mu}\right)^{k+l} \left\{ \alpha_k(P)\alpha_l(P')   -   \alpha_k(P')\alpha_l(P)\right\} \ge 0.
\end{equation}
Note that for  any $l<k$ we have
\begin{equation}
 \frac{\alpha_l(P')}{\alpha_k(P')} =  \prod\limits_{i = l+1}^k \frac{\alpha_{i-1}(P')}{\alpha_{i}(P')} \ge \prod\limits_{i = l+1}^k \frac{\alpha_{i-1}(P)}{\alpha_{i}(P)}  = \frac{\alpha_l(P)}{\alpha_k(P)}, 
\end{equation}
which concludes the proof.
\end{proof}

The condition~\eqref{eq:stoch_dom_BD_condition} on the edge selection probabilities can intuitively be justified when we compare the stationary distributions of the total number of jobs $Q(P')$ and $Q(P)$ to the stationary distributions of two birth-and-death processes with common arrival rates $N\lambda$ for all $q\ge 0$ and death rates $\mu\frac{\alpha_{q-1}(P')}{\alpha_{q}(P')}$ and $\mu\frac{\alpha_{q-1}(P)}{\alpha_{q}(P)}$ for all $q\ge 1$, respectively. Observe that the stationary distribution of the number of particles in the first birth-and-death process is given by
\begin{equation}
\pi'(q) = C' \frac{\left(N\lambda\right)^q}{\prod_{i=1}^q\mu\frac{\alpha_{i-1}(P')}{\alpha_{i}(P')}} = C'  \alpha_q(P')\left(\frac{N\lambda}{\mu}\right)^q = \mathbb{P}\{Q(P') = q\},
\end{equation}
a similar expression holds true for the second birth-and-death process.
When the death rates of the former process are always larger, due to~\eqref{eq:stoch_dom_BD_condition}, there is a higher drift toward the origin compared to the second process.
Hence, the corresponding birth-and-death process will yield fewer particles in the system, and also $Q(P')$ will be smaller than $Q(P)$.

\section{Simulation results for support of Conjecture~\ref{conj:dom_red}}\label{app:simulation}
We illustrate Conjecture~\ref{conj:dom_red} for larger systems and/or moderate load values by simulation results that compare the classical power-of-two policy with power-of-two policies for two specific graph structures:
\begin{enumerate}
\item The \textit{homogeneous ring} with $N$ servers as introduced in Subsection~\ref{subsec:intro_example}.
\item A \textit{homogeneous grid graph} with $N$ vertices (or servers) and homogeneous edge selection probabilities, with $N$ assumed to be a square. Each server can be identified with a tuple $(i,j)$, with $i,j=1,\dots,\sqrt{N}$. In the grid graph, server $(i,j)$ is connected to servers $(i-1,j), (i,j-1),(i+1,j)$ and $(i,j+1)$, where the indices are computed modulo $\sqrt{N}$. This yields $2N$ edges in total. Each existing edge $e = \{(i_1,j_1),(i_2,j_2)\}$ has an identical selection probability $p_e = 1/(2N)$.
\end{enumerate}
Settings with $N\in (9,49,196)$ servers and load $\rho \in (0.3, 0.5, 0.7, 0.9)$ are considered.
For each pair $(\rho,N)$, 50 simulation runs are conducted, each consisting of 10 million coupled arrival and departure events. The above constructions yield stable systems as long as $\rho<1$.

Figures~\ref{fig:coc_ring_errorbars} and~\ref{fig:coc_grid_errorbars} display the average difference between the cumulative distributions of $Q_N(P)$ and $Q^{*}_N$, i.e., the total number of jobs of a system with either the ring graph or the grid graph and a uniform complete graph, for various values of $N$ and $\rho$ as discussed above.
The widths of corresponding 95$\%$ confidence intervals are indistinguishable on the scale of the $y$-axis, and are hence omitted from the figures. 

It is clear that all settings support Conjecture~\ref{conj:dom_red}. Moreover, the benefit of the classical power-of-two setting becomes more profound for higher loads. This suggests that the results in Subsection~\ref{sec:main_results} hold beyond the light-traffic regime. However, it is known that the cumulative distributions of (scaled) $Q_N(P)$ and $Q^{*}_N$ vanishes in the heavy-traffic limit for $\rho\uparrow 1$~\cite[Theorem~1]{Cardinaels2022}, though this is not yet observable in these systems for pre-limit values.
Furthermore, the grid graph yields twice as many server pairs available for sampling compared to the ring graph, which translates into a smaller difference between the cumulative distributions of $Q_N(P)$ and $Q^{*}_N$ (note the difference in scales).

\begin{figure}[h]
\centering
\includegraphics[scale=1]{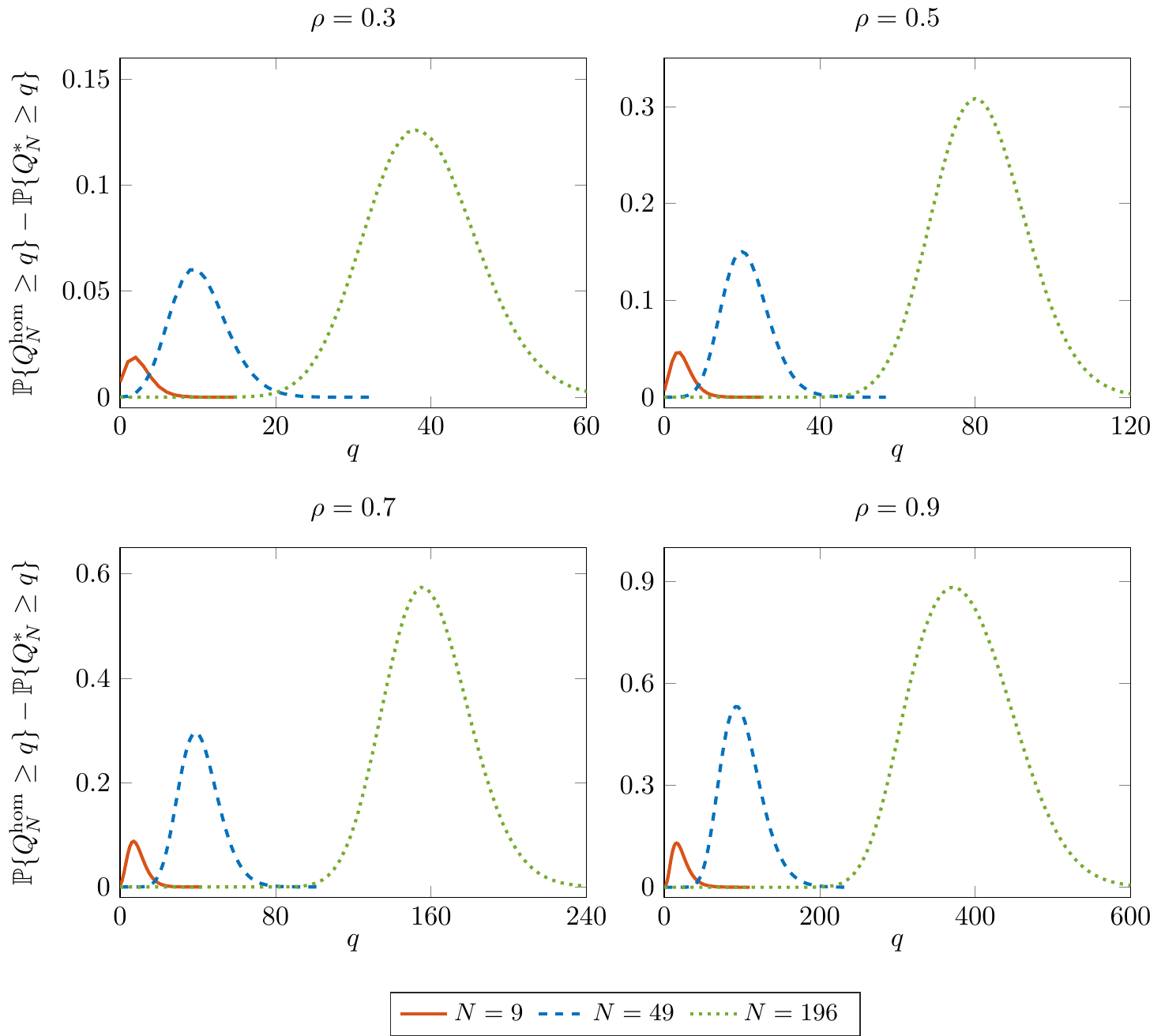}
\caption{A simulation based comparison between the stationary distribution of the total number of jobs for systems with the uniform complete graphs, $Q_N^*$, and the homogeneous ring, $Q_N^{\text{hom}}$, as underlying structures for various values of $N$ and $\rho$.}
\label{fig:coc_ring_errorbars}
\end{figure}

\begin{figure}[t]
\centering
\includegraphics[scale=1]{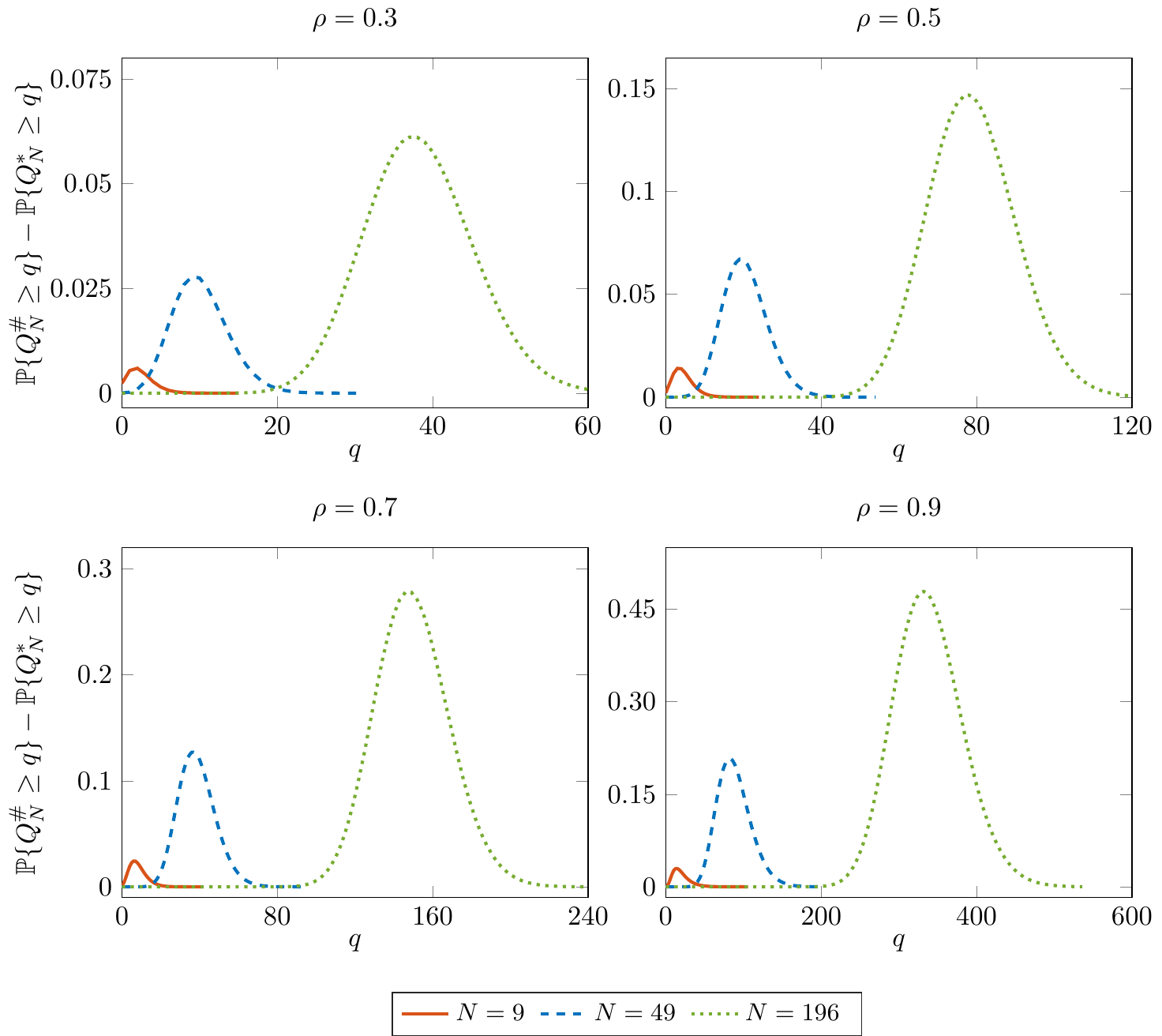}
\caption{A simulation based comparison between the stationary distribution of the total number of jobs for systems with the uniform complete graphs, $Q_N^*$, and the homogeneous grid, $Q_N^{\#}$, as underlying structures for various values of $N$ and $\rho$.}
\label{fig:coc_grid_errorbars}
\end{figure}

\section{Proofs of Lemmas \ref{lem:cos_complete} and \ref{lem:cos_ring}} \label{app:cos}
\begin{proof}[Proof of Lemma~\ref{lem:cos_complete}]
The main difference with the proof of Lemma~\ref{lem:uniform complete graph} is the fact that the state descriptor takes into account the types and order of the \textit{waiting} jobs and that the idle servers are incorporated in the product-form expression. For $q\ge 0$ we aim to find the stationary probabilities for the total number of jobs present in the system $Q^{\text{*,c.o.s.}}_4$, i.e.,
\begin{equation}\label{eq:sum_cos_compl}
\pr{Q^{\text{*,c.o.s.}}_4 = q} = \sum\limits_{(\boldsymbol{c},\boldsymbol{u})\in \mathcal{S}_q} \pi^{\text{c.o.s.}}(\boldsymbol{c},\boldsymbol{u})
\end{equation}
with 
$
\mathcal{S}_q \coloneqq \{ (\boldsymbol{c},\boldsymbol{u}) \in \mathcal{E}^{Q'}\times \{1,\dots,N\}^L \colon  Q'+(N-L) = q\} 
$
and the detailed stationary distribution $\pi^{\text{c.o.s.}}(\boldsymbol{c},\boldsymbol{u})$ as presented in~\eqref{eq:stat_distr_cos}.
When $q\ge 4$ there are three options, 2, 1 or 0 servers are idle, which partition $\mathcal{S}_q$ into three subsets. We investigate each part separately.

$|\boldsymbol{u}| = 2$:
If there are two idle servers, say $u_1$ and $u_2$, and waiting jobs present, then all $q-2$ waiting jobs must be of type $\{u_3,u_4\}$.
Hence, the contribution to~\eqref{eq:sum_cos_compl} is given by
\begin{equation}
4\cdot 3 \cdot C'\left(\frac{4\lambda\frac{1}{6}}{2\mu}\right)^{q-2}\frac{\mu}{4\lambda\frac{1}{2}}\frac{\mu}{4\lambda\frac{5}{6}} = C'\frac{1}{5} \left(\frac{\lambda}{3\mu}\right)^{q-4}.
\end{equation}

$|\boldsymbol{u}| = 1$. If there is a single idle server, say $u_1$, then the waiting jobs can be of any type expect for $\{u_1,u_i\}$ with $i=2,3,4$. This implies that the server rate sequence $(|c_1|,|c_1\cup c_2|,\dots,|c_1\cup\dots\cup c_{q-3}|)$ is of the form 
\[
(\underbrace{2,\dots,2}_{q_1},\underbrace{3,\dots,3}_{q-3-q_1})
\]
with $q_1 = 1,\dots,q-3$. Taking into account the number of states that yield the above server rate sequences, we find that the contribution to~\eqref{eq:sum_cos_compl} is given by
\begin{equation}
\begin{array}{rl}
& 4C'\left(\frac{4\lambda\frac{1}{6}}{\mu}\right)^{q-3}\frac{\mu}{4\lambda\frac{1}{2}}\left\{3\left(\frac{1}{2}\right)^{q-3} + 2\sum\limits_{k=1}^{q-4} \left(\frac{1}{2}\right)^{q_1}\left(\frac{3}{3}\right)^{q-3-q_1}\right\}\\
 
 = &C'\left(\frac{\lambda}{\mu}\right)^{q-4}\left\{4\left(\frac{2}{3}\right)^{q-3}-2\left(\frac{1}{3}\right)^{q-3}\right\}.
\end{array}
\end{equation}

$|\boldsymbol{u}| = 0$. If all servers are occupied, then the waiting jobs can be of any type. Hence, the server rate sequence $(|c_1|,|c_1\cup c_2|,\dots,|c_1\cup\dots\cup c_{q-4}|)$ is of the form 
\[
(\underbrace{2,\dots,2}_{q_1},\underbrace{3,\dots,3}_{q_2},\underbrace{4,\dots,4}_{q-4-q_1-q_2})
\]
with $q_1 = 1,\dots,q-4$ and $q_2 = 0,\dots, q-4-q_1$.
Taking into account the number of states that yield the above server rate sequences, we obtain that the contribution to~\eqref{eq:sum_cos_compl} amounts to 
\begin{equation}
\begin{array}{c}
 C'\left(\frac{4\lambda\frac{1}{6}}{\mu}\right)^{q-4}\left\{
6\left(\frac{1}{2}\right)^{q-4}
+8\sum\limits_{q_1 = 1}^{q-5} \left(\frac{1}{2}\right)^{q_1}\left(\frac{3}{3}\right)^{q-4-q_1}
+\sum\limits_{q_1 = 1}^{q-5} \left(\frac{1}{2}\right)^{q_1}\left(\frac{6}{4}\right)^{q-4-q_1} \dots \right.\\
 
\left.+4\sum\limits_{q_1 = 1}^{q-6} \sum\limits_{q_2 = 1}^{q-5-q_1} \left(\frac{1}{2}\right)^{q_1}\left(\frac{3}{3}\right)^{q_2}\left(\frac{6}{4}\right)^{q-4-q_1-q_2}
\right\}
 
 = C'\left(\frac{\lambda}{\mu}\right)^{q-4}\left\{\frac{1}{2}\left(\frac{1}{3}\right)^{q-4}-4\left(\frac{2}{3}\right)^{q-4} + \frac{9}{2}\right\}.
\end{array}
\end{equation}

This results in 
\begin{equation}
\pr{Q^{\text{*,c.o.s.}}_4 = q} =  C'\frac{9}{40} \left(\frac{\lambda}{\mu}\right)^{q-4}\left\{
20-\frac{5\cdot 2^5}{3^3} \left(\frac{2}{3}\right)^{q-4} +\frac{2^2}{3^3} \left(\frac{1}{3}\right)^{q-4}
\right\}
\end{equation}
for $q\ge 1$. Together with $\pr{Q^{\text{*,c.o.s.}}_4 = 0} =C'\left(\frac{\lambda}{\mu}\right)^{-4}$ this yields the stationary distribution~\eqref{eq:stat_complete_cos} after summation. This concludes the proof.
\end{proof}

\begin{proof}[Proof of Lemma~\ref{lem:cos_ring}] This proof follows the same reasoning as the the proof of Lemma~\ref{lem:cos_complete}, taking into account that there are now only four possible job types. With 
$
\mathcal{S}_q \coloneqq \{ (\boldsymbol{c},\boldsymbol{u}) \in \mathcal{E}^{Q'}\times \{1,\dots,N\}^L \colon  Q'+(N-L) = q\} 
$ we find that
\begin{equation}
\begin{array}{rcl}
\pr{Q^{\text{hom,c.o.s.}}_4 = q} & = & \sum\limits_{(\boldsymbol{c},\boldsymbol{u})\in\mathcal{S}^q} \pi^{\text{c.o.s.}}(\boldsymbol{c},\boldsymbol{u})\\

& = & 8C' \left(\frac{4\lambda\frac{1}{4}}{2\mu}\right)^{q-2}\frac{\mu}{4\lambda\frac{1}{2}}\frac{\mu}{4\lambda\frac{3}{4}}\\
& & +4C' \left(\frac{4\lambda\frac{1}{4}}{\mu}\right)^{q-3} \frac{\mu}{4\lambda\frac{1}{2}} \left\{2\left(\frac{1}{2}\right)^{q-3} + \sum\limits_{q_1 = 1}^{q-4} \left(\frac{1}{2}\right)^{q_1}\left(\frac{2}{3}\right)^{q-3-q_1} \right\}  \\

& & +C' \left(\frac{4\lambda\frac{1}{4}}{\mu}\right)^{q-4} \left\{
4\left(\frac{1}{2}\right)^{q-4} + 4\sum\limits_{q_1=1}^{q-5} \left(\frac{1}{2}\right)^{q_1}\left(\frac{2}{3}\right)^{q-4-q_1} + \sum\limits_{q_1=1}^{q-5} \left(\frac{1}{2}\right)^{q_1}\left(\frac{4}{4}\right)^{q-4-q_1}  \right.\\
& & \left. + \sum\limits_{q_1=1}^{q-6}\sum\limits_{q_2=1}^{q-5-q_1} \left(\frac{1}{2}\right)^{q_1}\left(\frac{2}{3}\right)^{q_2} \left(\frac{4}{4}\right)^{q-4-q_1-q_2} \right\}\\

& = & C'\left(\frac{\lambda}{\mu}\right)^{q-4} \left\{ 5+ \frac{1}{2}\left(\frac{2}{3}\right)^{q-4} -2 \left(\frac{2}{3}\right)^{q-4} \right\}
\end{array}
\end{equation}
for $q\ge 1$.
Together with $\pr{Q^{\text{hom,c.o.s.}}_4 = 0} =C'\frac{11}{48}\left(\frac{\lambda}{\mu}\right)^{-4}$ this yields the stationary distribution~\eqref{eq:stat_ring_cos} after summation. This concludes the proof.
\end{proof}

\end{appendices}

\end{document}